\documentclass[reqno, 12pt]{amsart}

\pdfoutput=1

\usepackage[numbers,sort,compress]{natbib}
\usepackage{ytableau}
\usepackage{enumerate}
\usepackage{latexsym}
\usepackage[centertags]{amsmath}
\usepackage{amsfonts}
\usepackage{amssymb}
\usepackage{mathabx}
\usepackage{amsthm}
\usepackage{newlfont}
\usepackage{graphics}
\usepackage{color}
\usepackage{xcolor}
\usepackage{float}
\usepackage{diagbox}
\usepackage{longtable}
\usepackage{rotating}
\usepackage{multirow}
\usepackage[hyphens]{url}
\usepackage{breakurl}
\usepackage{extarrows}
\usepackage{subfig}
\usepackage[sort,compress,numbers]{natbib}
\usepackage[utf8]{inputenc}

\newtheorem{thm}{Theorem}[section]
\newtheorem{cor}[thm]{Corollary}
\newtheorem{lem}[thm]{Lemma}
\newtheorem{prop}[thm]{Proposition}

\newtheorem{cnj}[thm]{Conjecture}
\newtheorem{dfn}[thm]{Definition}

\newtheorem{exa}[thm]{Example}
\newtheorem{prob}[thm]{Problem}

\allowdisplaybreaks[4]

\title{Symmetry of the refined $q,t$-Catalan polynomials for $\vec{k}$-Dyck paths}

\author{Menghao Qu$^{1}$ and Yingrui Zhang$^{2}$}
\address{$^{1}$Scuola Normale Superiore, Piazza dei Cavalieri 7, 56126 Pisa, Italy}
\address{$^{2}$Yunnan University of Finance and Economics, School of Statistics and Mathematics, 650221 Kunming, China}
\email{$^1$\texttt{menghao.qu@sns.it}\ \& $^2$\texttt{zyrzuhe@126.com}}

\begin{document}

\begin{abstract}
Pappe, Paul, and Schilling introduced two combinatorial statistics, depth and ddinv, associated with classical Dyck paths, and proved that the distributions of (area, depth) and (dinv, ddinv) are $q,t$-symmetric by constructing an involution on plane trees. They also provided a new formula for the original $q,t$-Catalan polynomials $C_{n}(q,t)$. We observe that depth is a slight modification of bounce, which was defined by the filling algorithm and ranking algorithm of Xin and the second author in their study of $\vec{k}$-Dyck paths. In this article, we generalize depth of classical Dyck paths to the case of $\vec{k}$-Dyck paths and prove $q,t$-symmetry of the pair of statistics (area, depth) for $\mathcal{K}$-Dyck paths. We provide an alternative description of the higher $q,t$-Catalan polynomials $C_{n}^{(k)}(q,t)$.
\end{abstract}

\maketitle

\noindent
\begin{small}
 \emph{Mathematic subject classification}: 05A19; 05C05; 05E10.
\end{small}

\noindent
\begin{small}
\end{small}

\section{Introduction}

$q,t$-Combinatorics is a branch of Combinatorics that studies the distribution of pairs of combinatorial statistics $(\mathrm{stat1, \mathrm{stat2}})$ on various objects. It plays an important role in the theory of symmetric functions, particularly in the theory of Macdonald polynomials \cite{macdonald1998symmetric}, which are $q,t$-generalizations of the Schur functions. The famous $q,t$-Catalan polynomials $C_{n}(q,t)$ originate from the study of diagonal harmonics \cite{garsia2002proof,haglund2005shuffle, haglund2008q}. It can be expressed as $C_n(q,t)=\left<\nabla e_n,e_n\right>$, where $e_n$ is the elementary symmetric function and $\nabla$ is a Macdonald eigenoperator \cite{bergeron1999identities}. The combinatorial formula for $C_n(q,t)$ is a sum over all classical Dyck paths $\mathcal{D}_{n}$, graded by pairs of statistics (area, bounce) or (dinv, area). The equivalence of these two expressions is verified by a bijection on $\mathcal{D}_{n}$, known as the zeta map \cite{haglund2008q} or sweep map \cite{armstrong2015sweep,thomas2018sweeping}.

A polynomial $F(q,t)$ in $q$ and $t$ is $q,t$-symmetric if $F(q,t)=F(t,q)$. The $q,t$-symmetry of $C_n(q,t)$ follows as a corollary of the famous shuffle theorem of Carlsson and Mellit \cite{carlsson2018proof}. However, a long-standing problem in the Algebraic Combinatorics community is to provide a direct combinatorial proof of the $q,t$-symmetry of $C_n(q,t)$, that is, to find a bijection $\phi$ from $\mathcal{D}_n$ to $\mathcal{D}_n$ such that $\mathrm{area}(\pi)=\mathrm{bounce}(\phi(\pi))$ and $\mathrm{bounce}(\pi)=\mathrm{area}(\phi(\pi))$.

There are generalizations of $C_n(q,t)$. In \cite{loehr2005conjectured}, the authors defined higher $q,t$-Catalan polynomials $C_{n}^{(k)}(q,t)=\left<\nabla^{k} e_n,e_n\right>$ and conjectured that it is the $q,t$-polynomial associated with pairs of statistics (area, bounce) on $k$-Dyck paths $\mathcal{D}_{k^n}$. The $q,t$-symmetry of $C_{n}^{(k)}(q,t)$ follows from the rational shuffle theorem of Mellit \cite{mellit2021toric}. However, a combinatorial proof of these $q,t$-symmetries remains an open problem. Naturally, some researchers have tried to extend these polynomials to other combinatorial objects. In \cite{d2022decorated}, the authors established a connection between Dyck paths and parallelogram polyominoes, defining similar statistics. In \cite{d2025shuffle}, they provided an expression of $C_n(q,t)$ in terms of sandpile models. In \cite{pappe2022area}, the authors introduced two new statistics, depth and ddinv, for classical Dyck paths $\mathcal{D}_{n}$ and proved that (area, depth) and (dinv, ddinv) are two $q,t$-symmetric pairs of statistics. The $q,t$-polynomial on classical Dyck paths $\mathcal{D}_{n}$, graded by (depth, ddinv), provides a new formula of $C_n(q,t)$.

To maintain brevity in the introduction, we will utilize certain notations that will be explained in detail in Section 2.

In this article, we focus on a generalization of classical Dyck paths $\mathcal{D}_{n}$ and
$k$-Dyck paths $\mathcal{D}_{k^n}$, called $\vec{k}$-Dyck paths $\mathcal{D}_{\vec{k}}$, where $\vec{k}$ is a vector of positive integers. This reduces to classical Dyck paths when $\vec{k}=(1,1,\cdots,1)$ and to $k$-Dyck paths when $\vec{k}=(k,k,\cdots,k)$. Xin and the second author introduced three statistics, dinv, area, and bounce for these objects and initiated the study of $q,t$-symmetry of $C_{\vec{k}}(q,t)$ and $C_{\mathcal{K}}(q,t)$, defined by pairs of statistics (dinv, area) or (area, bounce) on $\vec{k}$-Dyck paths. In a series of articles on this topic \cite{xin2023dinv, niu,chen,xin2025q}, the authors proved that $C_{\vec{k}}(q,t)$ is $q,t$-symmetric for $\ell(\vec{k})\leq 3$, $\vec{k}=(k,k,k,k)$, and $\vec{k}=(k,k,k,k,k)$. They also showed that $C_{\mathcal{K}}(q,t)$ is $q,t$-symmetric for $\ell(\lambda(\mathcal{K}))\leq 3$ and $\lambda(\mathcal{K})=((a+1)^s,a^{4-s})$, where $a$ and $s$ are positive integers. The $q,t$-symmetry of these polynomials was derived from their generating functions, which can be computed using explicit formulas for the bounce statistic and simplified with tools from MacMahon's Partition Analysis. However, the $q,t$-symmetry of $C_{\mathcal{K}}(q,t)$ does not generally hold for $\ell(\lambda(\mathcal{K}))\geq 4$.

Our main results are as follows: We generalize the depth statistic of classical Dyck paths, as introduced in \cite{pappe2022area}, to the case of $\vec{k}$-Dyck paths. The $q,t$-polynomials $\widetilde{C}_{\vec{k}}$ and $\widetilde{C}_{\mathcal{K}}(q,t)$, graded by (area,depth) exhibit good $q,t$-symmetry.

\begin{thm}\label{thm-aK}
$\widetilde{C}_{^a\mathcal{K}}(q,t)$ is $q,t$-symmetric for any $\vec{k}$ and positive $a$.
\end{thm}

The fact that the sum of $q,t$-symmetric polynomials is also $q,t$-symmetric will lead to

\begin{thm}\label{thm-K}
$\widetilde{C}_{\mathcal{K}}(q,t)$ is $q,t$-symmetric for any $\vec{k}$.
\end{thm}

We also propose two conjectures on $q,t$-symmetry for $\vec{k}$-Dyck paths.

\begin{cnj}
$\widetilde{C}_{\mathcal{K}^b}(q,t)$ is $q,t$-symmetric for any $\vec{k}$ and positive $b$.
\end{cnj}

\begin{cnj}
$\widetilde{C}_{^a\mathcal{K}^b}(q,t)$ is $q,t$-symmetric for any $\vec{k}$, and positive $a$ and $b$.
\end{cnj}

\begin{exa}
Computer data show that $\widetilde{C}_{(1,1,3,1)}(q,t)$ and $\widetilde{C}_{(1,3,1,1)}(q,t)$ are not $q,t$-symmetric. However $\widetilde{C}_{(1,1,3,1)}(q,t)+\widetilde{C}_{(1,3,1,1)}(q,t)$ is $q,t$-symmetric since
$$\{(1,1,3,1),(1,3,1,1)\}=^1\{(1,3),(3,1)\}^1,$$ and $\{(1,3),(3,1)\}$ represents the set of all rearrangements of $\vec{k}=(1,3)$ or $(3,1)$. Since $\widetilde{C}_{(1,1,1,3)}(q,t)$ is $q,t$-symmetric, it follows that $\widetilde{C}_{(1,1,3,1)}(q,t)+\widetilde{C}_{(1,3,1,1)}(q,t)+\widetilde{C}_{(1,1,1,3)}(q,t)$ is also $q,t$-symmetric, as we have the following expression: $$\{(1,1,3,1),(1,3,1,1),(1,1,1,3)\}=^1\{(1,3,1),(3,1,1),(1,1,3)\}.$$
\end{exa}

The paper is organized as follows. In Section 2, we provide the necessary definitions of $\vec{k}$-Dyck paths and review some previous combinatorial statistics on them.
In Section 3, we introduce the depth statistic for $\vec{k}$-Dyck paths and define our $q,t$-polynomials $\widetilde{C}_{\vec{k}}(q,t)$ and $\widetilde{C}_{\mathcal{K}}(q,t)$. We discuss the $q,t$-symmetry of $\widetilde{C}_{\mathcal{K}}(q,t)$ in Section 4. Our proof constructs an involution that interchanges the area and depth based on a dual algorithm on the set of labeled branch trees $\mathcal{LBT}_{\mathcal{K}}$, which is bijective with $\mathcal{D}_{\mathcal{K}}$. In Section 5, we explore the relationship between $\widetilde{C}_{\vec{k}}(q,t)$ and $C_{\vec{k}}(q,t)$ for some special choices of $\vec{k}$. We further investigate the case of $\vec{k}$ with $\ell(\vec{k})\leq 3$ and obtain their $q,t$-symmetry in two ways: first, by providing a more direct involution, and second, by analyzing their generating functions, as the authors do in \cite{xin2025q}. Finally, we discuss some further directions for the study of $q,t$-symmetry in $\vec{k}$-Dyck paths.

\section{Background and Definitions}

We primarily follow the definitions and notations from \cite{xin2019sweep} and \cite{xin2023dinv}.

\subsection{Two models of $\vec{k}$-Dyck paths}

Given a vector $\vec{k}=(k_1,k_2,\cdots,k_{\ell})$ of positive integers, we denote by $\ell(\vec{k}):=\ell$ the \textbf{length} of $\vec{k}$, and $|\vec{k}|:=k_1+k_2+\cdots+k_{\ell}$ the \textbf{size} of $\vec{k}$. Such $\vec{k}$ is also called a \textbf{composition} or \textbf{ordered partition}.

\begin{dfn}[Visual path model]
A \textbf{$\vec{k}$-Dyck path} is a lattice path from $(0,0)$ to $(|\vec{k}|+\ell(\vec{k}),0)$ that never goes below the horizontal axis with up steps (red arrows) $(1,k_i)$, $1\leq i\leq \ell(\vec{k})$ from left to right and down steps (blue arrows) $(1,-1)$.
\end{dfn}

\begin{dfn}[Word model]
We can identify a $\vec{k}$-Dyck path $\pi$ with its \textbf{$SW$-word} $\pi=\pi_1\pi_2\cdots \pi_{|\vec{k}|+\ell(\vec{k})}$, where each $\pi_i$ is either $S^{k_j}$ or $W$, depending on whether the $i$-th vertex of $\pi$ corresponds to the $j$-th \textbf{South end} (of the $j$-th North step) or the \textbf{West end} (of an East step).
\end{dfn}

\begin{exa}
\end{exa}
\begin{figure}[H]
    \centering
    \includegraphics[width=0.6\linewidth]{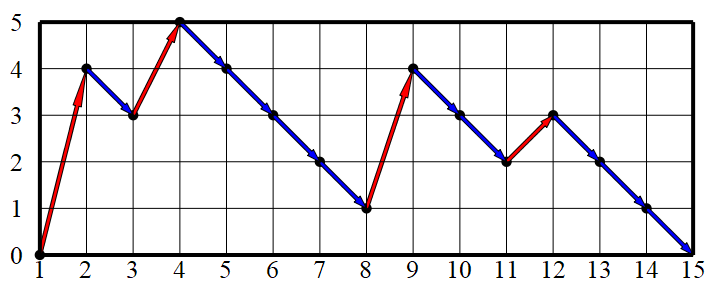}
    \caption{$\vec{k}=(4,2,3,1)$ and $\pi=S^4 W S^2 W W W W S^3 W W S^1 W W W$.}
    \label{fig-kDyck}
\end{figure}

Denote by $\mathcal{D}_{\vec{k}}$ the set of all $\vec{k}$-Dyck paths. Let $\mathcal{K}$ be the set of all rearrangements of $\vec{k}$. We refer to the paths in $\mathcal{D}_{\mathcal{K}}$, the union of all $\mathcal{D}_{\vec{k}}$, as $\mathcal{K}$-Dyck paths. Let $\lambda(\mathcal{K}):=\lambda(\vec{k})$ denote the common partition obtained by ordering the entries of $\vec{k}$ in decreasing order.

For any $\vec{k}$, and positive $a$ and $b$, we define three operations on $\mathcal{K}$.
\begin{align*}
^a\mathcal{K}:&=\{(a,j_1,j_2,\cdots,j_{\ell}): \ (j_1,j_2,\cdots,j_{\ell})\in \mathcal{K}\},\\
\mathcal{K}^{b}:&=\{(j_1,j_2,\cdots,j_{\ell},b): \ (j_1,j_2,\cdots,j_{\ell})\in \mathcal{K}\},\\
^a\mathcal{K}^{b}:&=\{(a,j_1,j_2,\cdots,j_{\ell},b):\ (j_1,j_2,\cdots,j_{\ell})\in \mathcal{K}\}.
\end{align*}

\begin{exa}
If $\vec{k}=\{1,3,2\}$, then
\begin{align*}
\mathcal{K}&=\{(1,2,3),(1,3,2),(2,1,3),(2,3,1),(3,1,2),(3,2,1)\},\\
^4\mathcal{K}&=\{(4,1,2,3),(4,1,3,2),(4,2,1,3),(4,2,3,1),(4,3,1,2),(4,3,2,1)\},\\
\mathcal{K}^{2}&=\{(1,2,3,2),(1,3,2,2),(2,1,3,2),(2,3,1,2),(3,1,2,2),(3,2,1,2)\},\\
^4\mathcal{K}^2&=\{(4,1,2,3,2),(4,1,3,2,2),(4,2,1,3,2),(4,2,3,1,2),(4,3,1,2,2),(4,3,2,1,2)\}.
\end{align*}
\end{exa}

If $\vec{k}=(1,1,\cdots,1)$, then $\mathcal{D}_{\mathcal{K}}:=\mathcal{D}_{n}$ is the set of all classical Dyck paths. If $\vec{k}=(k,k,\cdots,k)$, then $\mathcal{D}_{\mathcal{K}}:=\mathcal{D}_{k^n}$ is the set of all $k$-Dyck paths. In particular,
\begin{equation*}
\left|\mathcal{D}_{n}\right|=\frac{1}{n+1}\binom{2n}{n}, \quad \text{and } \left|\mathcal{D}_{k^n}\right|=\frac{1}{kn+1}\binom{(k+1)n}{n},
\end{equation*}
which are called the $n$-th \textbf{Catalan number} and \textbf{Fuss-Catalan number} respectively.

In \cite{rukavicka2011generalized}, the author provided an explicit formula for $|\mathcal{D}_{\mathcal{K}}|$. Let $\lambda(\mathcal{K})=1^{m_1}2^{m_2}\cdots |\vec{k}|^{m_{|\vec{k}|}}$, where $m_i$ represents the multiplicity of part $i$. Then, the number of all $\mathcal{K}$-Dyck paths is given by the following expression:
\begin{equation*}
\left|\mathcal{D}_{\mathcal{K}}\right|=\frac{1}{|\vec{k}|+1}\binom{|\vec{k}|+m_1+\cdots+m_{|\vec{k}|}}{|\vec{k}|, m_1,\cdots,m_{|\vec{k}|}}.
\end{equation*}

\subsection{Combinatorial statistics of $\vec{k}$-Dyck paths}

We associate each arrow with a \textbf{starting rank} (resp. \textbf{ending rank}). The starting rank of each arrow is defined by setting $r_1 = 0$ for the first arrow of $\pi$. For $ 1\leq i \leq |\vec{k}|+\ell(\vec{k})-1$, the starting rank $r_{i+1}$ is recursively assigned as follows: If the $i$-th arrow $\pi_i=S^{k_j}$, then $r_{i+1}:=r_{i}+k_j$; Otherwise, if $\pi_i=W$, then $r_{i+1}:=r_{i}-1$. It is exactly equal to the $y$-coordinate of the starting point of each arrow. The \textbf{starting rank sequence} of $\pi\in\mathcal{D}_{\vec{k}}$ is denoted by $r(\pi)=(r_1,r_2,\cdots,r_{|\vec{k}|+\ell(\vec{k})})$. Similarly, the \textbf{ending rank} of each arrow is defined to be the $y$-coordinate of its ending point. The \textbf{ending rank sequence} is denoted by $\dot{r}(\pi)=(\dot{r}_1,\dot{r}_2,\cdots,\dot{r}_{|\vec{k}|+\ell(\vec{k})}))$. It is clear that $\dot{r}(\pi)=(r_2,\cdots,r_{|\vec{k}|+\ell(\vec{k})},0)$.

\begin{exa}
In Figure \ref{fig-kDyck},
\begin{align*}
r(\pi)=(0,4,3,5,4,3,2,1,4,3,2,3,2,1),\\
\dot{r}(\pi)=(4,3,5,4,3,2,1,4,3,2,3,2,1,0).
\end{align*}
\end{exa}

\begin{dfn}
Given $\pi\in \mathcal{D}_{\vec{k}}$, the \textbf{area sequence} of $\pi$ is $a(\pi):=(a_1,a_2,\cdots,a_{\ell(\vec{k})})$, where $a_i$ denotes the starting rank of the $i$-th red arrow. Define the \textbf{area} of $\pi$ by $$\mathrm{area}(\pi):=a_1+\cdots+a_{\ell(\vec{k})}.$$
\end{dfn}

Geometrically, the area of $\pi$ counts the
number of complete squares between the start of red arrows and the horizontal axis. It is also the number of complete squares between the lattice path and the main diagonal, in rows containing a south end of a north step.

\begin{exa}
In Figure 1, we have $a(\pi)=(0,3,1,2)$ and $\mathrm{area}(\pi)=6$.
\end{exa}

Given $\pi\in \mathcal{D}_{\vec{k}}$, we use the notation $W_i$ (resp. $S_j$) to refer to the $i$-th blue (resp. $j$-th red) arrow in $\pi$. For any two arrows $A$ and $B$ in $\pi$, we say that $A<B$ if $A$ is to the left of $B$ in $\pi$.

The \textbf{dinv} consists of two parts, which can be described geometrically as follows: The first summand, called \textbf{sweep dinv}, is the number of pairs
$(W_i,S_j)$ with $W_i<S_j$ such that $S_j$ starts at a rank weakly below the start of $W_i$ and ends strictly above the end of $W_i$. The second part, called \textbf{red dinv}, specifies that each pair of red arrows $(S_i,S_j)$ with $S_i<S_j$ contributes $|\dot{r}(S_i)-\dot{r}(S_j)|$ if $S_j$ starts at a rank weakly below the start of $S_i$ and ends strictly above the end of $S_i$, or if $S_j$ starts at a rank strictly above the start of $S_i$ and ends strictly below the end of $S_i$.

\begin{dfn}
Given $\pi\in \mathcal{D}_{\vec{k}}$,
\begin{align*}
dinv(\pi)=\sum_{W_i<S_j}\chi(0\leq r(W_i)&-r(S_j)\leq k_j)\\
&+\sum_{S_i<S_j}\chi(r(S_i)\geq r(S_j) \ \& \ \dot{r}(S_j)>\dot{r}(S_i))(\dot{r}(S_j)-\dot{r}(S_i))\\
&+\sum_{S_i<S_j}\chi(r(S_i)< r(S_j) \ \& \ \dot{r}(S_j)<\dot{r}(S_i))(\dot{r}(S_i)-\dot{r}(S_j)).
\end{align*}
\end{dfn}

\begin{exa}
In Figure \ref{fig-kDyck}, all pairs of arrows contributing to the sweep dinv are
\begin{align*}
(W_1,S_2),(W_1,S_3), (W_3,S_3), (W_4,S_3),(W_4,S_4),(W_5,S_3),(W_5,S_4),(W_7,S_4).
\end{align*}
$(S_1,S_4)$ contributes 1 red dinv, and $(S_3,S_4)$ contributes 1 red dinv. Thus $\mathrm{dinv}(\pi)=10$.
\end{exa}

To define \rm{bounce}, we need two algorithms. The first is called the \textbf{Filling Algorithm $\eta$}, which was originally introduced by Garsia and Xin \cite{garsia2020sweep} for Fuss rational Dyck paths and later generalized to the case of $\vec{k}$-Dyck paths in \cite{xin2019sweep}.

Let $\mathcal{F}_{\vec{k}}$ denote the set of all filling tableaux where the
$i$-th column contains $k_i+1$ entries for $1\leq i\leq \ell(\vec{k})$.  Each tableau is a bijective filling with labels $1,2,\cdots,|\vec{k}|+\ell(\vec{k})$, such that the entries in each row are increasing from left to right and the entries in each column are increasing from top to bottom. Furthermore, for any $a<b<c<d$, where $d$ is immediately below $a$, then the labels $b$ and $c$ cannot appear in the same column.

\begin{dfn}\cite[\textbf{Filling Algorithm $\eta$}]{xin2019sweep}
\label{al-Filling Algorithm}
\noindent
Input: The SW-word of a $\vec{k}$-Dyck path $\pi\in \mathcal{D}_{\vec{k}}$.
\noindent
Output: A filling tableau $\eta(\pi)\in \mathcal{F}_{\vec{k}}$.

\begin{enumerate}
\item   Start by placing a $1$ in the top row and the first column.
\item  If the second letter in $\pi$ is an $S^*$, place a $2$ at the top of the second column.
\item   If the second letter in $\pi$ is a $W$, place $2$ below the $1$.
\item  At any stage, the entry at the bottom of the $i$-th column but not in row $k_i+1$ will be called \textbf{active}.
\item  Having placed $1,2,\cdots i-1$, place $i$ immediately below the \textbf{smallest  active} entry if the $i$-th letter in $\pi$ is a $W$; otherwise, place $i$ at the top of the first empty column.
\item   Repeat this process recursively until $1,2,\cdots ,|\vec{k}|+\ell(\vec{k})$ have  all  been placed.
\end{enumerate}
\end{dfn}

\begin{thm}\cite[Theorem 2.6]{xin2019sweep}
The \textbf{Filling Algorithm $\eta$} defines a bijection from $\mathcal{D}_{\vec{k}}$ to $\mathcal{F}_{\vec{k}}$.
\end{thm}

Similarly, extending the idea of Garsia and Xin from \cite{garsia2020sweep}, the second author and Xin introduced the second algorithm, called \textbf{Ranking Algorithm $\gamma$}.

\begin{dfn}\cite[\textbf{Ranking Algorithm $\gamma$}]{xin2019sweep}\label{dfn-ranking}
Input: A filling tableau $F\in \mathcal{F}_{\vec{k}}$.
Output: A ranking tableau $\gamma(F)$ of the same shape with $F$.
\begin{enumerate}
\item Successively assign ranks $0,1,2,\cdots,k_1$ to the first column of $\gamma(F)$ from top to bottom;
\item For $i$ from 2 to $n$, if the top entry of the $i$-th column of $F$ is $A + 1$, and the rank of $A$ is $a$, then the ranks in the $i$-th column of $\gamma(F)$ are successively $a, a + 1, ..., a + k_i$ from top to bottom.
\end{enumerate}
\end{dfn}

\begin{dfn}
Given $\pi\in \mathcal{D}_{\vec{k}}$, define the \textbf{bounce sequence} $b(\pi):=(b_1,b_2,\cdots,b_{\ell(\vec{k})})$ as the entries in the first row of $\gamma(\eta(\pi))$ obtained by applying \textbf{Filling Algorithm $\eta$} and \textbf{Ranking Algorithm $\gamma$}. The \textbf{bounce} of $\pi$ is defined by
\begin{equation*}
\mathrm{bounce}(\pi):=b_1+b_2+\cdots+b_{\ell(\vec{k})}.
\end{equation*}
\end{dfn}

\begin{exa}\label{exa-bounce}
The following tableaux are obtained by applying the Filling Algorithm $\eta$ and Ranking Algorithm $\gamma$ to $\pi=S^4 W S^2 W W W W S^3 W W S^1 W W W$. We have $b(\pi)=(0,1,3,4)$ and $\mathrm{bounce}(\pi)=8$.

\begin{figure}[H]
\centering
\begin{ytableau}
1 & 3 & 8 & 11\\
2 & 5 & 10 & 13\\
4 & 7 & 12 & \none\\
6 & \none & 14 & \none\\
9
\end{ytableau}\qquad
\begin{ytableau}
0 & 1 & 3 & 4\\
1 & 2 & 4 & 5\\
2 & 3 & 5 & \none\\
3 & \none & 6 & \none\\
4
\end{ytableau}
\vspace{0.5cm}
\caption{Filling tableau $\eta(\pi)$ and Ranking tableau $\gamma(\eta(\pi))$.}
\end{figure}
\end{exa}

\begin{dfn}
\textbf{Sweep map} $\Phi$: Sorting the two-line array according to the starting rank sequence from smallest to largest. For arrows starting at the same rank, we list them from right to left.
\end{dfn}

\begin{exa}
In Figure 1, we have
\begin{equation*}
\left(\begin{matrix}
\pi\\r(\pi)
\end{matrix}\right)=
\left(\begin{matrix}
S^4 &W &S^2 & W &W &W &W &S^3 &W &W &S^1 &W &W &W\\
0 & 4 & 3 & 5 & 4 & 3 & 2 & 1 & 4 & 3 & 2 & 3 & 2 & 1
\end{matrix}\right)
\end{equation*}
Applying the sweep map $\Phi$ to the above array, we obtain
\begin{equation*}
\left(\begin{matrix}
S^4 & W & S^3 & W & S^1 & W & W & W & W& S^2 & W & W & W & W\\
0 & 1 & 1 & 2 & 2 & 2 & 3 & 3 & 3 & 3 & 4 & 4 & 4 & 5
\end{matrix}\right),
\end{equation*}
thus
$$\Phi(\pi)=(S^4 W S^3 W S^1 W W W W S^2 W W W W).$$
\end{exa}

From our example, it is clear that $\Phi$ is not a map from $\mathcal{D}_{\vec{k}}$ to $\mathcal{D}_{\vec{k}}$, as we have rearranged the order of all $S^*$ steps. However

\begin{thm}\cite[Theorem 1]{xin2023dinv}
Sweep map $\Phi$ is a bijection from $\mathcal{D}_{\mathcal{K}}$ to $\mathcal{D}_{\mathcal{K}}$ and takes dinv to area and area to bounce, that is $\mathrm{dinv}(\pi)=\mathrm{area}(\Phi(\pi))$ and $\mathrm{area}(\pi)=\mathrm{bounce}(\Phi(\pi))$.
\end{thm}

This establishes that the three statistics are jointly equidistributed on $\mathcal{K}$-Dyck paths.

\begin{cor}\cite[Theorem 1]{xin2023dinv}
\begin{equation}\label{equ-areabounce}
\sum_{\pi\in \mathcal{D}_{\mathcal{K}}}q^{\mathrm{dinv}(\pi)}t^{\mathrm{area}(\pi)}=\sum_{\pi\in \mathcal{D}_{\mathcal{K}}}q^{\mathrm{area}(\pi)}t^{\mathrm{bounce}(\pi)}.
\end{equation}
\end{cor}

In particular, there are two different expressions for the classical $q,t$-Catalan polynomials $C_{n}(q,t)$ and the higher $q,t$-Catalan polynomials $C_{n}^{(k)}(q,t)$.

\begin{cor}\cite{garsia2002proof,loehr2005conjectured,mellit2021toric}
\begin{equation}
C_{n}(q,t)=\sum_{\pi\in \mathcal{D}_{n}}q^{\mathrm{dinv(\pi)}}t^{\mathrm{area}(\pi)}=\sum_{\pi\in \mathcal{D}_{n}}q^{\mathrm{area(\pi)}}t^{\mathrm{bounce}(\pi)}.
\end{equation}
\begin{equation}\label{equ-cnk}
C_{n}^{(k)}(q,t)=\sum_{\pi\in \mathcal{D}_{k^n}}q^{\mathrm{dinv(\pi)}}t^{\mathrm{area}(\pi)}=\sum_{\pi\in \mathcal{D}_{k^n}}q^{\mathrm{area(\pi)}}t^{\mathrm{bounce}(\pi)}.
\end{equation}
\end{cor}

As a consequence of the Shuffle Theorem \cite{carlsson2018proof} and the Rational Shuffle Theorem \cite{mellit2021toric}, both polynomials exhibit $q,t$-symmetry.

\section{Refined $q,t$-Catalan polynomials of $\vec{k}$-Dyck paths}

A natural question arises: Are the polynomials defined in Equation (\ref{equ-areabounce}), which generalize both the classical $q,t$-Catalan polynomials and the higher $q,t$-Catalan polynomials, $q,t$-symmetric? If not, are there any pair of statistics that exhibit $q,t$-symmetry on $\mathcal{D}_{\vec{k}}$, $\mathcal{D}_{\mathcal{K}}$, or on certain partial unions of the set $\mathcal{D}_{\vec{k}}$?

We begin by outlining some established results on $q,t$-symmetric polynomials that are closely related to our work.

\subsection{Area-Bounce Polynomials}
\begin{dfn}
\begin{equation}
C_{\vec{k}}(q,t):=\sum_{\pi\in \mathcal{D}_{\vec{k}}}q^{\mathrm{area}(\pi)}t^{\mathrm{bounce}(\pi)}.
\end{equation}
\begin{equation}
C_{\mathcal{K}}(q,t):=\sum_{\pi\in \mathcal{D}_{\mathcal{K}}}q^{\mathrm{area}(\pi)}t^{\mathrm{bounce}(\pi)}.
\end{equation}
\end{dfn}

As mentioned in the introduction, there are certain special cases where $C_{\vec{k}}(q,t)$ and $C_{\mathcal{K}}(q,t)$ are $q,t$-symmetric, including cases such as $\ell(\vec{k})\leq 3$, $\vec{k}=(k,k,k,k)$, and others. However, the $q,t$-symmetry of $C_{\mathcal{K}}(q,t)$ no longer holds in general for $\ell(\lambda(\mathcal{K}))\geq 4$. The experimental results suggest that

\begin{cnj}\cite[Conjecture 11]{xin2023dinv}
$C_{\mathcal{K}}(q,t)$ is $q,t$-symmetric if $\lambda(\mathcal{K})=((a+1)^s,a^{m})$ for any positive $a$, $s$, and $m$.
\end{cnj}

In \cite{beck2024polyhedral}, the authors discussed the $q,t$-symmetry of $C_{\vec{k}}(q,t)$ and $C_{\mathcal{K}}(q,t)$ for some of the special cases mentioned above, approaching the problem from the perspective of polyhedral geometry. They also proposed a conjecture for a broad family of $\vec{k}$-Dyck paths.

\begin{cnj}\cite[Conjecture 5.1]{beck2024polyhedral}\label{cnj-beck}
$C_{\vec{k}}(q,t)$ is $q,t$-symmetric if $\vec{k}=(a,b,b,\cdots,b)$ for any positive integers $a$ and $b$.
\end{cnj}

They proved that $C_{\vec{k}}(q,t)$ is not influenced by the last part of $\vec{k}$, that is

\begin{prop}\cite[Corollary 6.2]{beck2024polyhedral}\label{thm-bk1k2}
For any two vectors $\vec{k}_1=(a_1,a_2,\cdots,a_{\ell},b)$ and $\vec{k}_2=(a_1,a_2,\cdots,a_{\ell},c)$ with any positive integers $a_1, \cdots,a_{\ell},b$, and $c$,  $C_{\vec{k}_1}(q,t)=C_{\vec{k}_2}(q,t)$.
\end{prop}

Hence Conjecture \ref{cnj-beck} can be also stated as
\begin{cnj}\cite[Conjecture 5.1]{beck2024polyhedral}
$C_{\vec{k}}(q,t)$ is $q,t$-symmetric if $\vec{k}=(a,b,b,\cdots,b,c)$ for any positive integers $a$, $b$, and $c$.
\end{cnj}

\subsection{Area-Depth polynomials}

The content of this section is the main focus of our study. First, let us review the depth statistic for classical Dyck paths.

\begin{dfn}\cite[\textbf{Depth labeling}]{pappe2022area}
Given $\pi\in \mathcal{D}_{n}$, label $\pi$ column by column using the following algorithm:
\begin{enumerate}
    \item In the first column, label all cells directly to the right of a $N$ step with a $0$;
    \item In the $i$-th column from the left, locate the bottommost cell $c$ in the column that is directly right of a North step; note that such a cell may not exist. From $c$ travel Southwest diagonally
    until a cell $c'$ that is already labeled is reached. Let $\ell$ be the labeling of $c'$. Label all cells directly to the right of a North step in the $i$-th column with an $\ell+1$.
\end{enumerate}
The \textbf{depth labeling sequence} $d(\pi):=(d_1,d_2,\cdots,d_n)$ is the sequence of labels read from bottom to top and define the \textbf{depth} of $\pi$ by $\mathrm{depth}(\pi):=d_1+d_2+\cdots+d_n$.
\end{dfn}

\begin{exa}
In Figure \ref{exa-Dyckdepth}, we have $d(\pi)=(0,0,1,1,1,2,2,2)$ and $\mathrm{depth}(\pi)=9$.
\begin{figure}[H]
    \centering
    \includegraphics[width=0.35\linewidth]{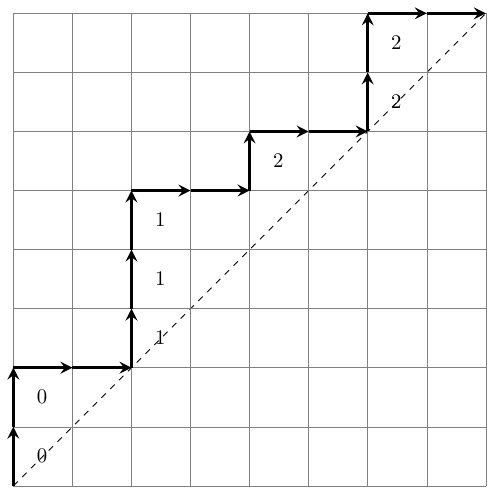}
    \caption{Depth labeling of $\pi=NNEENNNEENEENNEE$.}
    \label{exa-Dyckdepth}
\end{figure}
\end{exa}

Before defining the depth for $\vec{k}$-Dyck paths, we would like to take a moment to explain our motivation. We observed that if we modified the \textbf{smallest active} to \textbf{largest active} in the Filling algorithm $\eta$, the first row of the ranking tableau will correspond to the \textbf{depth labeling sequence}. It should be noted that \textbf{depth labeling sequence} differs from the \textbf{depth sequence} defined in \cite{pappe2022area}; the latter is simply a rearrangement of the former.

The filling tableaux obtained via the modified filling algorithm $\eta_{*}$ are denoted by $\mathcal{NF}_{\vec{k}}$. $\mathcal{NF}_{\mathcal{K}}$ is the union of all $\mathcal{NF}_{\vec{k}}$. A characterization of these tableaux will be provided later.

\begin{dfn}[\textbf{Filling Algorithm $\eta_{*}$}]
\label{dfn-Filling Algorithm-Max}
\noindent
Input: The SW-word of a $\vec{k}$-Dyck path $\pi\in \mathcal{D}_{\vec{k}}$.
\noindent
Output: A filling tableau $\eta_{*}(\pi)\in \mathcal{NF}_{\vec{k}}$.

\begin{enumerate}
\item   Start by placing a $1$ in the top row and the first column.
\item  If the second letter in $\pi$ is an $S^*$, place a $2$ at the top of the second column.
\item   If the second letter in $\pi$ is a $W$, place $2$ below the $1$.
\item  At any stage, the entry at the bottom of the $i$-th column but not in row $k_i+1$ will be called \textbf{active}.
\item  Having placed $1,2,\cdots i-1$, place $i$ immediately below the \textbf{largest  active} entry if the $i$-th letter in $\pi$ is a $W$; otherwise, place $i$ at the top of the first empty column.
\item   Repeat this process recursively until $1,2,\ldots ,|\vec{k}|+l(\vec{k})$ have  all  been placed.
\end{enumerate}
\end{dfn}

We simply refer to the ranking algorithm as $\gamma_{*}$ when applying the same procedure $\gamma$ described in Definition \ref{dfn-ranking} to tableaux in $\mathcal{NF}_{\vec{k}}$.

\begin{exa}
In the context of $\vec{k}$-Dyck paths, the path $\pi$ in Example \ref{exa-Dyckdepth} can be expressed as
$$\pi=S^1 S^1 W W S^1 S^1 S^1 W W S^1 W W S^1 S^1 W W.$$
The first row in its ranking tableau is exactly depth labeling sequence of $\pi$.
\begin{figure}[H]
\centering
\begin{ytableau}
1 & 2 & 5 & 6 & 7 & 10 & 13 & 14\\
4 & 3 & 12 & 9 & 8 & 11 & 16 & 15
\end{ytableau}\qquad
\begin{ytableau}
0 & 0 & 1 & 1 & 1 & 2 & 2 & 2\\
1 & 1 & 2 & 2 & 2 & 3 & 3 & 3
\end{ytableau}
\vspace{0.5cm}
\caption{Filling tableau $\eta_{*}(\pi)$ and Ranking tableau $\gamma_{*}(\eta_{*}(\pi))$.}
\end{figure}
\end{exa}

\begin{lem}\label{lem-NFprop}
Let $F$ be a tableau with labels $1,2,\dots,|\vec{k}|+\ell(\vec{k})$, such that the $i$-th column contains $k_i+1$ entries for $1 \leq i \leq \ell(\vec{k})$, the entries in the first row are increasing from left to right, and each column is increasing from top to bottom. Then $F = \eta_*(\pi)$ for some $\pi \in \mathcal{D}_{\vec{k}}$ if and only if, for any $a<d-1$, where $d$ is immediately below $a$, the columns of $F$ containing $a+1,a+2,\dots, d-1$ are filled. 
\end{lem}
\begin{proof}
The proof follows similarly to that of Lemma 2.7 in \cite{garsia2020sweep}.

The ``only if" direction is immediate. Since $d$ is placed directly below $a$, it follows that $a$ became active as soon as it was placed and remained active until $d$ arrived. If there are still unlabeled cells in the columns of $F$ that contain the labels $a+1,a+2,\dots, d-1$, then according to the placement rules, $d$ would have been placed below $d-1$, not $a$, leading to a contradiction.

The ``if" part of the proof is more involved. Given a tableau $F$ satisfying the stated conditions, we aim to show that $F$ can be obtained from some $\pi \in \mathcal{D}_{\vec{k}}$ using the filling algorithm $\eta_*$. Assume the first row of $F$ is $t_1, t_2, \dots, t_{\ell(\vec{k})}$. If we have proven that
$$t_{j+1} \leq k_1 + k_2 + \cdots + k_j + j + 1,$$ 
for all  $1 \leq j \leq \ell(\vec{k})-1$, then we define $\pi$ to be the $\vec{k}$-Dyck path with letters $S^{k_1},S^{k_2},\dots, S^{k_{\ell(\vec{k})}}$ placed at positions $t_1, t_2, \dots, t_{\ell(\vec{k})}$ respectively. 

Note that there are exactly $k_1 + k_2 + \cdots + k_j + j$ cells in the first $j$ columns of $F$. Given this, if for some $j$ we had $t_{j+1} > k_1 + k_2 + \cdots + k_j + j + 1$, then the increasing conditions on both the first row and all columns would not leave enough space for any entries $a < t_{j+1}$, leading to a contradiction.

It remains to show that the entries in $F$ are placed as they would be by our filling algorithm $\eta_*$. Suppose that the entries of $F$ are placed one by one, in increasing order, as we find them. The increasing conditions on both the first row and all columns force each entry to be placed directly under some active entry. We need to show that this active entry is the larger one. Suppose, for the sake of contradiction, that the entry $d$ is placed under an active entry $a$ that is smaller than the larger active entry $d-1$ at that moment. This would result in $a<d-1$, and there exist unlabeled cells in the columns of $F$ that contain the labels $a+1,a+2,\dots, d-1$, leading to a contradiction. 

This completes the proof.
\end{proof}

As shown in the above example, the resulting filling tableau $\eta_{*}(\pi)\in \mathcal{NF}_{\vec{k}}$ does not belong to $\mathcal{F}_{\vec{k}}$. However, the transition from the smallest active to the largest active is clearly well defined. Note that for a given $\vec{k}$-Dyck path $\pi$, both the filling algorithms $\eta$ and $\eta_{*}$ applied to $\pi$ yield the same first row, corresponding to the indices of $S^{k_j}$ in $\eta(\pi)$ and $\eta_*(\pi)$. Since the entries in the first row uniquely determine the entire filling tableau, the same statements in \cite[Lemma 2.5 and Theorem 2.6]{xin2019sweep} yields

\begin{prop}
The \textbf{Filling Algorithm} $\eta_{*}$ defines a bijection from $\mathcal{D}_{\vec{k}}$ to $\mathcal{NF}_{\vec{k}}$.
\end{prop}

\begin{dfn}
Given $\pi\in \mathcal{D}_{\vec{k}}$, define the \textbf{depth labeling sequence} $d(\pi):=(d_1,d_2,\cdots,d_{\ell})$ as the entries in its first row of $\gamma_{*}(\eta_{*}(\pi))$ obtained by applying \textbf{Filling Algorithm $\eta_*$} and \textbf{Ranking Algorithm $\gamma_*$}. Then the \textbf{depth} of $\pi$ is $$\mathrm{depth}(\pi):=d_1+d_2+\cdots+d_{\ell}.$$
\end{dfn}

\begin{exa}\label{exa-depth}
The following tableaux are obtained by applying Filling Algorithm $\eta_*$ and Ranking Algorithm $\gamma_*$ to $\pi=S^4 W S^2 W W W W S^3 W W S^1 W W W$, we have $d(\pi)=(0,1,3,4)$ and $\mathrm{depth}(\pi)=0+1+3+5=9$.
\begin{figure}[H]
\centering
\begin{ytableau}
1 & 3 & 8 & 11\\
2 & 4 & 9 & 12\\
6 & 5 & 10 & \none\\
7 & \none & 13 & \none\\
14
\end{ytableau}\qquad
\begin{ytableau}
0 & 1 & 3 & 5\\
1 & 2 & 4 & 6\\
2 & 3 & 5 & \none\\
3 & \none & 6 & \none\\
4
\end{ytableau}
\vspace{0.5cm}
\caption{Filling tableau $\eta_{*}(\pi)$ and Ranking tableau $\gamma_{*}(\eta_{*}(\pi))$.}
\end{figure}
\end{exa}

By comparing the values of bounce and depth in Example \ref{exa-bounce} and Example \ref{exa-depth}, we find that these are two distinct statistics. Thus, we introduce two new $q,t$-polynomials

\begin{dfn}
\begin{equation}
\widetilde{C}_{\vec{k}}(q,t):=\sum_{\pi\in \mathcal{D}_{\vec{k}}}q^{\mathrm{area}(\pi)}t^{\mathrm{depth}(\pi)},
\end{equation}
\begin{equation}
\widetilde{C}_{\mathcal{K}}(q,t):=\sum_{\pi\in \mathcal{D}_{\mathcal{K}}}q^{\mathrm{area}(\pi)}t^{\mathrm{depth}(\pi)}.
\end{equation}
\end{dfn}

In \cite{pappe2022area}, it was proven that $\widetilde{C}_{(1,1,\cdots,1)}(q,t)$ is $q,t$-symmetric, which covers the case of classical Dyck paths. We will investigate their $q,t$-symmetry in other cases in the next section.

\section{$q,t$-symmetry of $\widetilde{C}_{\mathcal{K}}(q,t)$}

In this section, we will prove our main result, Theorem \ref{thm-K}, by first establishing its generalization, Theorem \ref{thm-aK}. Our approach is to construct an involution  $\omega$ that interchanges area and depth. The involution is based on a duality involving a combinatorial object $\mathcal{LBT}_ {\mathcal{K}}$, which we refer to as \textbf{labeled branch trees}.

We present a simple flowchart illustrating the involution $\omega$ on $\mathcal{D}_{\mathcal{K}}$.
$$\omega: \mathcal{D}_{\mathcal{K}} \xrightarrow[\text{Definition }\ref{dfn-Filling Algorithm-Max}]{\text{Filling algorithm } \eta_*} \mathcal{NF}_{\mathcal{K}}
\xrightarrow[\text{Definition }\ref{dfn-tableaux2tree}]{\delta} \mathcal{LBT}_ {\mathcal{K}}
\xrightarrow[\text{Definition }\ref{dfn-dual}]{dual} \mathcal{LBT}_ {\mathcal{K}}
\xrightarrow{ \delta^{-1}} \mathcal{NF}_{\mathcal{K}}
\xrightarrow{ \eta_{*}^{-1}} \mathcal{D}_{\mathcal{K}}.$$
Both of $\eta_*$ and
$\delta$ are bijections.

\subsection{Labeled branch trees}
We first need to review some definitions from graph theory.
\begin{dfn}
A \textbf{tree} is a connected graph containing no cycles. A \textbf{rooted tree} is a tree in which one vertex has been designated the root. A \textbf{plane tree} is a rooted tree in which the children of each node are linearly ordered.
\end{dfn}

In diagrams, we usually keep the root at the top and list other vertices below it. There are many bijections between $\mathcal{D}_{n}$ and $\mathcal{T}_{n+1}$, where $\mathcal{T}_{n+1}$ denotes the set of all plane trees with $n+1$ nodes. Three such examples can be found in \cite{pappe2022area}.

Now, we aim to establish a bijection between $\mathcal{D}_{\vec{k}}$ and a family of labeled plane trees $\mathcal{LBT}_{\vec{k}}$. For this, we define some terminology based on the concepts from \cite{pappe}.

\begin{dfn}
A \textbf{leaf} is any vertex having no children. An \textbf{extended leaf} is an unlabeled path graph with exactly one end-vertex designated as the leaf. The \textbf{length} of an extended leaf $E$, denoted by $\ell(E)$, is the number of edges in $E$.
The \textbf{top vertex} of $E$ is the vertex farthest from its leaf.
\end{dfn}

\begin{dfn}\cite{pappe}
For any $T\in \mathcal{T}_{n+1}$, there exists a unique \textbf{extended leaf decomposition}. Let $v_i$ be the $i$-th leaf, as read from left to right, of a plane tree $T$ with $m$ leaves. For each leaf $v_i$, we trace the path from $v_i$ to the closer of the two:
\begin{enumerate}
\item the root, or
\item the closest ancestor of $v_i$ that has two or more children, and $v_i$ is not the leftmost of those children nor a descendant of the leftmost child.
\end{enumerate}
Each path corresponds to an extended leaf.
\end{dfn}

\begin{exa}
In Figure \ref{fig-extendedleaf}, $T$ have four leaves and four extended leaves. The red path is an extended leaf of lenth $4$ in $T$,  and its top vertex is the root of $T$.
\begin{figure}[H]
    \centering
    \includegraphics[width=0.3\linewidth]{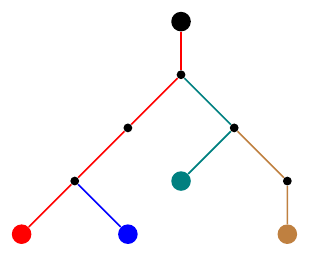}
    \caption{The extended leaf decomposition of a plane tree $T\in \mathcal{T}_{10}$.}
    \label{fig-extendedleaf}
\end{figure}
\end{exa}

We extend the definition of an extended leaf and refer to it as a \textbf{branch}. At the same time, we introduce the notion of \textbf{branch decomposition} for a tree.

\begin{dfn}
A \textbf{branch} of a tree $T \in \mathcal{T}_{n+1}$ is a path graph of any length from one vertex to its descendants.
The \textbf{length} of a branch $B$, denoted by $\ell(B)$, is the number of edges in $B$. The branch $B$ contains $\ell(B)+1$ vertices, and the \textbf{top vertex} of $B$ is the vertex closest to the root of the tree $T$.
\end{dfn}

Note that a branch does not necessarily contain a leaf of the tree $T$.

\begin{dfn}
The \textbf{branch decomposition} of a tree $T\in \mathcal{T}_{n+1}$ is the partitioning of the tree into a specific number of branches, subject to three constraints:
\begin{enumerate}
\item Each vertex (expect for the top vertex) on a branch is the leftmost descendant of its top vertex;
\item The edges of the branches do not intersect;
\item The union of the edge sets of all branches is exactly the edge set of the original tree $T$.
\end{enumerate}
We say that the branch decomposition of $T$, which is still denoted by $T$ when there is no confusion, is a \textbf{branch tree}.
\end{dfn}

It is immediate that the following facts hold: A branch $B$ with length $\ell(B) \ge 2$ can be further decomposed into two branches; Each leaf of a tree belongs to a different branch; A tree $T \in \mathcal{T}_{n+1}$ can be decomposed into at most $n$ branches, and at least as many branches as the number of leaves in the tree.

Preorder traversal follows a `root-left-right' sequence: the algorithm visits the current node before recursively exploring its children. Given a branch tree $T$ with $\ell$ branches, we refer to the branches  $B_1,B_2,\cdots,B_{\ell}$ as those obtained by traversing the tree in preorder, with the branch $B_1$ called the \textbf{initial branch}. We say $B_j$ is a child of $B_i$ if the top vertex of $B_j$ is one of vertices of $B_i$. If several different branches, such as $B_i$, $B_j$, and  $B_k$, share a common top vertex, and $B_i$ is to the left of $B_j$, and $B_j$ is to the left of $B_k$, we say that $B_j$ is a child of $B_i$, and $B_k$ is a child of $B_j$. It is easy to observe that each vertex in the branch is connected to at most one child. We always place a child to the right of its parent.

Next, we present a labeling algorithm for a branch tree $T$, which is described recursively.
For each branch $B_i$ with length $k_i:=\ell(B_i)$, place one \textbf{red label} on the left side of the branch $B_i$ and $k_i$ \textbf{blue labels} on its right side (one blue label on each edge on the right side). If an edge is the $m$-th edge starting from the top vertex of $B_i$, we call it the $m$-th edge of $B_i$. For each edge $E$ on the right side of $B_i$, there are two vertices. We call the vertex of edge $E$ that is closest to the top vertex of $B_i$ the \textbf{start-vertex}, and the other vertex the \textbf{end-vertex}. 

\begin{dfn}[\textbf{Labeling Algorithm $\mathcal{L}$}]\label{dfn-Labeling}

Input: A branch tree $T \in \mathcal{T}_{n+1}$ with $\ell$ branches.

Output: A labeled branch tree $\mathcal{L}(T)$ with labeling entries $1,2,\cdots,\ell+n$.

\noindent As a general rule, when labeling any branch, its left side is always processed first, followed by its right side. Furthermore, the labeling of the right side strictly begins with the edge connected to the top vertex of the branch. The specific steps are as follows:
\begin{enumerate}
\item  Mark the number $1$ on the left side of the initial branch $B_1$. If the top vertex of $B_1$ is connected to a child branch $B_i$, then mark the number $2$ on the left side of the branch $B_i$. Otherwise, continue marking the edges with numbers $2,3,\dots, m+1$ until the end-vertex of the $m$-th edge of $B_1$ is connected to a child branch.

\item Assume that we mark an edge as $k$ and its end-vertex is connected to a child branch $B_j$. In that case, mark the left side of $B_j$ with $k+1$.

\item Assume that we mark the left side of $B_j$ with $k$:
   \begin{enumerate}
   \item If the top vertex is connected to a child branch $B_p$, mark the number $k+1$ on the left side of $B_p$;
   \item Otherwise, mark the edges with numbers $k+1,k+2,\dots, k+m$ until the end-vertex of the $m$-th edge of $B_j$ is connected to a child branch. Then, return to step $(2)$.
   \end{enumerate}

\item Once the marking of the right side of the branch is complete (and it does not belong to step $(2)$), return to the parent branch and resume marking the unmarked segments on its right side.
\end{enumerate}
\end{dfn}

For a labeled branch tree $\mathcal{L}(T)$ with $\ell$ branches, we partition it into $\ell$ labeled path graphs $\tilde{B}_1,\tilde{B}_2,\cdots,\tilde{B}_\ell$ ordered from smallest to largest based on the values of their red labels. From the labeled path graphs $\tilde{B}_1,\tilde{B}_2,\cdots,\tilde{B}_\ell$, it is straightforward to reconstruct the labeled branch tree. Note that the order of the $\tilde{B}_i$'s may differ from that of the original branches $B_i$.

Given any $\vec{k}$, define the set of \textbf{labeled branch tree of type $\vec{k}$} by
\begin{align*}
 \mathcal{LBT}_{\vec{k}}:=  \{\mathcal{L}(T): T \in \mathcal{T}_{|\vec{k}|+1}\text{ and } \ell(\tilde{B_i}) = k_i \text{ for all } 1\leq i\leq \ell(\vec{k}) \}.
\end{align*}
Similar to other notations, we use $\mathcal{LBT}_{\mathcal{K}}$ to denote the union of all $\mathcal{LBT}_{\vec{k}}$ for $\vec{k}\in \mathcal{K}$. 

When restricting the set, we will still use $T$ to denote a labeled branch tree.

\begin{exa}
In Figure \ref{fig-labledbranchtree}, there are four labeled path graphs $\tilde{B}_1$ (red), $\tilde{B}_2$ (blue), $\tilde{B}_3$ (brown), and $\tilde{B}_4$ (teal). $\ell(\tilde{B}_1)=1$, $\ell(\tilde{B}_2)=2$, $\ell(\tilde{B}_3)=1$ and $\ell(\tilde{B}_4)=1$. $\tilde{B}_2$ is a child of $\tilde{B}_1$, $\tilde{B}_3$ is a child of $\tilde{B}_2$, and $\tilde{B}_4$ is a child of $\tilde{B}_3$.

\begin{figure}[H]
    \centering
    \includegraphics[width=0.4\linewidth]{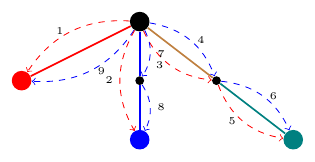}
    \caption{A labeled branch tree $ T \in \mathcal{LBT}_{(1,2,1,1)}$.}
    \label{fig-labledbranchtree}
\end{figure}
\end{exa}

We now construct a bijection between the set of filling tableaux $\mathcal{NF}_{\vec{k}}$ and the set of labeled branch tree $\mathcal{LBT}_{\vec{k}}$ for any vector $\vec{k} = (k_1,k_2,\dots,k_{\ell})$.

\begin{dfn}[\textbf{The map} $\delta$]\label{dfn-tableaux2tree}
Given $F\in \mathcal{NF}_{\vec{k}}$, we associate $F$ with $\ell$ labeled path graphs $\tilde{B}_1,\tilde{B}_2,\cdots,\tilde{B}_{\ell}$.
The $i$-th column of $F$ corresponds to a labeled path graph $\tilde{B}_i$ of length $k_i$, with the red label as the first entry and the blue labels as the remaining $k_i$ entries, arranged from top to bottom. Define $\delta(F)$ to be the labeled branch tree obtained from the labeled path graphs $\tilde{B}_1,\tilde{B}_2,\cdots,\tilde{B}_{\ell}$.
\end{dfn}

\begin{prop}\label{prop-mapdelta}
The map $\delta$ defines a bijection from $\mathcal{NF}_{\vec{k}}$ to $\mathcal{LBT}_{\vec{k}}$.
\end{prop}
\begin{proof}
For a labeled branch tree $T \in \mathcal{LBT}_{\vec{k}}$ with $\ell$ labeled path graphs $\tilde{B}_1,\tilde{B}_2,\cdots,\tilde{B}_\ell$, we associate $T$ with a tableau $F$ consisting of $\ell$ columns, where the $i$-th column has $k_i+1$ cells. For $1\leq i\leq \ell$, we place the red label of $\tilde{B}_i$ in the first row, column $i$, and the blue labels of $\tilde{B}_i$ in the remaining $k_i$ cells, arranged from top to bottom. Then, by the Labeling algorithm $\mathcal{L}$, the filling tableau $F$ satisfies the conditions in Lemma \ref{lem-NFprop}.
Therefore, we have $F \in \mathcal{NF}_{\vec{k}}$. In fact, the above describes the inverse of the map $\delta$. The proof is complete.
\end{proof}

\begin{dfn}[\textbf{Dual Algorithm $\mathrm{dual}$}]\label{dfn-dual}
Input: a labeled branch tree $T\in \mathcal{LBT}_{\mathcal{K}}$.
Output: a labeled branch tree $T^{\mathrm{dual}}\in \mathcal{LBT}_{\mathcal{K}}$.
\begin{enumerate}
    \item Perform the branch decomposition to $T$ and reorder the resulting elements according to the increasing red labels, yielding a sequence of labeled path graphs $\tilde{B}_1,\tilde{B}_2,\cdots,\tilde{B}_{\ell}$.
    \item For each $\tilde{B}_j$ and its parent $\tilde{B}_i$ with $i<j$, if the length from the top vertex of $\tilde{B}_j$ to the top vertex of $\tilde{B}_i$ is $m$, reattach $\tilde{B}_j$ to $\tilde{B}_i$ such that the new length from the top vertex of $\tilde{B}_j$ to the top vertex of $\tilde{B}_i$ is $\ell(\tilde{B}_i)-m$. 
    \item Apply the labeling algorithm $\mathcal{L}$ to the (unlabeled) branch tree obtained in step $(2)$.
\end{enumerate}
\end{dfn}

\subsection{Proof of Theorem \ref{thm-aK}.}

Let $\tilde{B}_1,\tilde{B}_2,\cdots,\tilde{B}_{\ell}$ be the sequence of labeled path graphs of $T\in \mathcal{LBT}_{\mathcal{K}}$, ordered increasingly according to their red labels. We associate two sequences of nonnegative integers, $a(T)=(a_1,a_2,\cdots,a_{\ell})$ and $d(T)=(d_1,d_2,\cdots,d_{\ell})$ with $\tilde{B}_1,\tilde{B}_2,\cdots,\tilde{B}_{\ell}$. First, we set $\tilde{B}_1$ to be associated with $a_1:=0$. Then for each $\tilde{B}_j$ and its parent $\tilde{B}_i$ with some $i<j$, we recursively associate $a_j:=a_i+e_{ij}$, where $e_{ij}$ is the length from the top vertex of $\tilde{B}_j$ to the leaf (ending node) of $\tilde{B}_i$. Similarly, we set $\tilde{B}_1$ to be associated with $d_1:=0$. Then for each $\tilde{B}_j$ and its parent $\tilde{B}_i$ with $i<j$, we recursively associate $d_j:=d_i+s_{ij}$, where $s_{ij}$ is the length from the top vertex of $\tilde{B}_j$ to the top vertex (starting node) of $\tilde{B}_i$.

We define the \textbf{area} and \textbf{depth} of $T\in \mathcal{LBT}_{\vec{k}}$ by
\begin{align*}
\mathrm{area}(T)&:=a_1+a_2+\cdots+a_{\ell},\\
\mathrm{depth}(T)&:=d_1+d_2+\cdots+d_{\ell}.
\end{align*}

\begin{prop}\label{prop-areadepth}
Given $T\in \mathcal{LBT}_{\mathcal{K}}$, then we have $\pi=\eta_{*}^{-1}(\delta^{-1}(T)) \in \mathcal{D}_{\mathcal{K}}$. Furthermore, $a(T)=(a_1,a_2,\cdots,a_{\ell})$ is exactly the area sequence of $\pi$, and $d(T)=(d_1,d_2,\cdots,d_{\ell})$ is exactly the depth labeling sequence of $\pi$. Thus, we have $\mathrm{area}(T)=\mathrm{area}(\pi)$ and $\mathrm{depth}(T)=\mathrm{depth}(\pi)$.
\end{prop}

Before we begin proving the proposition, let us first review the definitions of area and depth for $\pi \in \mathcal{D}_{\vec{k}}$. The starting rank sequence $r(\pi) = (r_1,r_2,\dots,r_{|\vec{k}|+\ell(\vec{k})})$ is obtained recursively by setting $r_1 = 0$, and for $1\leq i\leq |\vec{k}|+\ell(\vec{k})-1$, we define $r_{i+1} = r_i + k_j$ if the $i$-th letter $\pi_i = S^{k_j}$, or $r_{i+1} = r_i-1$ if $\pi_i = W$. The area sequence of $\pi$ is $a(\pi) = (a_1,a_2,\dots,a_{\ell})$, where $a_i$ is the starting rank of the $i$-th red arrow ($S^{*}$ step). The area of $\pi$ is given by area($\pi) = a_1 + a_2 + \cdots + a_{\ell}$. The depth labeling sequence $d(\pi) = (d_1,d_2,\dots,d_{\ell})$ consists of the entries in the first row of $\gamma_*(\eta_*(\pi))$, and $\mathrm{depth}(\pi)=d_1 + d_2 + \cdots + d_{\ell}$.

In the following lemma, we provide equivalent definitions of the area sequence and the depth sequence in terms of the filling tableau $F$.

\begin{lem}\label{lem-anotherareadepth}
Given $\pi \in \mathcal{D}_{\vec{k}}$, let $t(F) = (t_1,t_2,\dots,t_{\ell})$ be the first row entries of the filling tableau $F = \eta_*(\pi)$. For each entry $t_j$, where $2 \leq j \leq \ell$, we can assume that $t_j - 1$ is located in column $i$ with $i<j$, and let $m_{ij}$ denote its row index in $F$. Then, the area sequence $a(\pi) = (a_1,a_2,\dots,a_{\ell})$ satisfies $a_1 = 0$, and for $j\geq 2$, $a_j = a_{i} + k_i - m_{ij} + 1$. The depth labeling sequence $d(\pi) = (d_1,d_2,\dots,d_{\ell})$ satisfies $d_1 = 0$, and for $j\geq 2$, $d_j = d_{i} + m_{ij}-1$.
\end{lem}

\begin{proof}
We will prove the result case by case.

It is easy to observe that the $t_i$-th letter $\pi_{t_i}$ is $S^{k_i}$ for $1 \leq i \leq \ell$. If, for any entry $t_j$, the entry $t_j - 1$ is in row $1$, that is, $m_{ij}=1$, then $i$ must be equal to $j-1$. In this case, we have $a_j = a_{i} + k_i$ by the definition of the starting rank sequence.

Otherwise, suppose $t_j-1$ is located in column $i$ and row $m_{ij} (\geq 2)$. Assume that the first $m_{ij}$ entries in the $i$-th column of $F$ are $s_1, s_2,\cdots, s_{m_{ij}}$. It is clear that $s_1=t_i$ and $s_{m_{ij}}=t_j-1$. We have $\pi_{s_1}=S^{k_i}$ and $\pi_{s_p}=W$ for $2\leq p\leq m_{ij}$. Now, we will consider the rank subsequence $(r_{s_1}, r_{s_2},\cdots, r_{s_{m_{ij}}})$ in the starting rank sequence $r(\pi)$. We have $a_i = r_{s_1}$ and $a_j = r_{s_{m_{ij}}} -1$. It remains to show that $r_{s_{m_{ij}}} = r_{s_1} + k_i-m_{ij}+2$. If $s_1, s_2,\cdots, s_{m_{ij}}$ are consecutive numbers, i.e., $s_{p}= s_{p-1}+1$ for $2\leq p\leq m_{ij}$, then we have $r_{s_2}=r_{s_1}+k_i$ and $r_{s_p}=r_{s_{p-1}}-1$ for $3\leq p\leq m_{ij}$. Thus, $r_{s_{m_{ij}}} = r_{s_1} + k_i-m_{ij}+2$. Otherwise, assume that for some $2\leq z\leq m_{ij}$, we have $s_{z}>s_{z-1}+1$. By Lemma \ref{lem-NFprop}, we can conclude that the entries $s_{z-1}+1, s_{z-1}+2, \cdots, s_{z}-1$ occupy some of entire columns of $F$. In other words, in the SW-subword $(\pi_{s_{z-1}+1}, \pi_{s_{z-1}+2}, \cdots, \pi_{s_{z}-1})$, the number of $W$ is equal to the sum of $k_s$ where $k_s$ denotes the index of the column that the entries $s_{z-1}+1, s_{z-1}+2, \cdots, s_{z}-1$ occupy in $F$. Therefore, we also have $r_{s_2}=r_{s_1}+k_i$ and $r_{s_{p}} = r_{s_{p-1}} - 1$ for $3\leq p\leq m_{ij}$, and thus $r_{s_{m_{ij}}} = r_{s_1}+k_i-m_{ij}+2$. This concludes the analysis of the area sequence.

The statement about the depth labeling sequence is straightforward. By the Ranking algorithm, $d_j$ is the rank of $t_j$, which is equal to the rank of $t_j - 1$. Since the entry $t_j - 1$ is located in column $i$ and row $m_{ij}$ in $F$, its rank is $d_i+m_{ij}-1$. Therefore, we have $d_j=d_i+m_{ij}-1$.
\end{proof}

\begin{proof}[Proof of the Proposition \ref{prop-areadepth}]
Given $\pi \in \mathcal{D}_{\vec{k}}$, let $F=\eta_{*}(\pi)$ be the filling tableau with the first row entries $t(F) = (t_1,t_2,\dots,t_{\ell})$. By Proposition \ref{prop-mapdelta}, there exists a unique $T=\delta(F)\in \mathcal{LBT}_{\vec{k}}$. The $i$-th column of $F$ corresponds to a labeled path graph $\tilde{B}_i$ of lenght $k_i$ in $T$. By the definitions of the area and depth of $T$ and Lemma \ref{lem-anotherareadepth}, we only need to show that $e_{ij} =k_i-m_{ij}+1$ and $s_{ij}=m_{ij}-1$. This result follows trivially from the structure of the labeled path graphs, completing the argument.
\end{proof}

\begin{cor}
Given $T\in \mathcal{LBT}_{\mathcal{K}}$, we have $(T^{\mathrm{dual}})^{\mathrm{dual}}=T$. In particular,
\begin{equation*}
\mathrm{area}(T^{\mathrm{dual}})=\mathrm{depth}(T) \quad \text{and} \quad \mathrm{depth}(T^{\mathrm{dual}})=\mathrm{area}(T).
\end{equation*}
\end{cor}
\begin{proof}
Dual algorithm does not alter the parent-child relationship between any two branches of $T$. The statement follows directly from the definitions of the area and depth of $T$.
\end{proof}

\begin{cor}\label{cor-omega}
Given $\pi\in \mathcal{D}_{\mathcal{K}}$, $\omega(\pi):=\eta_{*}^{-1}\circ\delta^{-1}\circ(\delta\circ\eta_{*}(\pi))^{\mathrm{dual}}\in \mathcal{D}_{\mathcal{K}}$ is an involution interchanging area and depth, that is 
\begin{equation*}
\mathrm{area}(\pi)=\mathrm{depth}(\omega(\pi)) \ \text{ and }\ \mathrm{depth}(\pi)=\mathrm{area}(\omega(\pi)).    
\end{equation*}
Furthermore $\omega$ keeps the first part of $\vec{k}\in \mathcal{K}$ unchanged, that means the first step is always $S^{k_1}$. Consequently, $\omega$ restricts to an involution on $\mathcal{D}_{^a\mathcal{K}}$ that similarly interchanges area and depth.
\end{cor}

This completes the proof of our main result, Theorem \ref{thm-aK}. A natural consequence of Corollary \ref{cor-omega} is that this provides a new interpretation of the higher $q,t$-Catalan polynomials.

\begin{cor}
\begin{equation}
C_{n}^{(k)}(q,t)=\sum_{\pi\in \mathcal{D}_{k^n}}q^{\mathrm{depth(\pi)}}t^{\mathrm{dinv(\omega(\pi))}}.
\end{equation}
\end{cor}
\begin{proof}
This follows from Equation (\ref{equ-cnk}) and the fact that $\omega$ takes area to depth.
\end{proof}

\begin{exa}
The Figure below illustrates the process from $\pi$ to $\omega(\pi)$.
The second row provides an example of the construction algorithm $\delta$, while the third row demonstrates its inverse. Let $T$ be $\delta(\eta_{*}(\pi))$. To obtain $a(T)=(0,3,1,2)$, we have $a_1=0$, $a_2=a_1+3=3$, $a_3=a_1+1=1$, and $a_4=a_3+1=2$. Similarly, we obtain $d(T)=(0,1,3,5)$, with $d_1=0$, $d_2=d_1+1=1$, $d_3=d_1+3=3$, and $d_4=d_3+2=5$. It is easy to check that $a(T^{\mathrm{dual}})=(0,3,5,1)$ and $d(T^{\mathrm{dual}})=(0,1,2,3)$. Therefore, $\mathrm{area}(\pi)=\mathrm{depth}(\omega(\pi))=6$, and $\mathrm{depth}(\pi)=\mathrm{area}(\omega(\pi))=9$.
\end{exa}

\begin{figure}[H]
    \centering
    \includegraphics[width=1\linewidth]{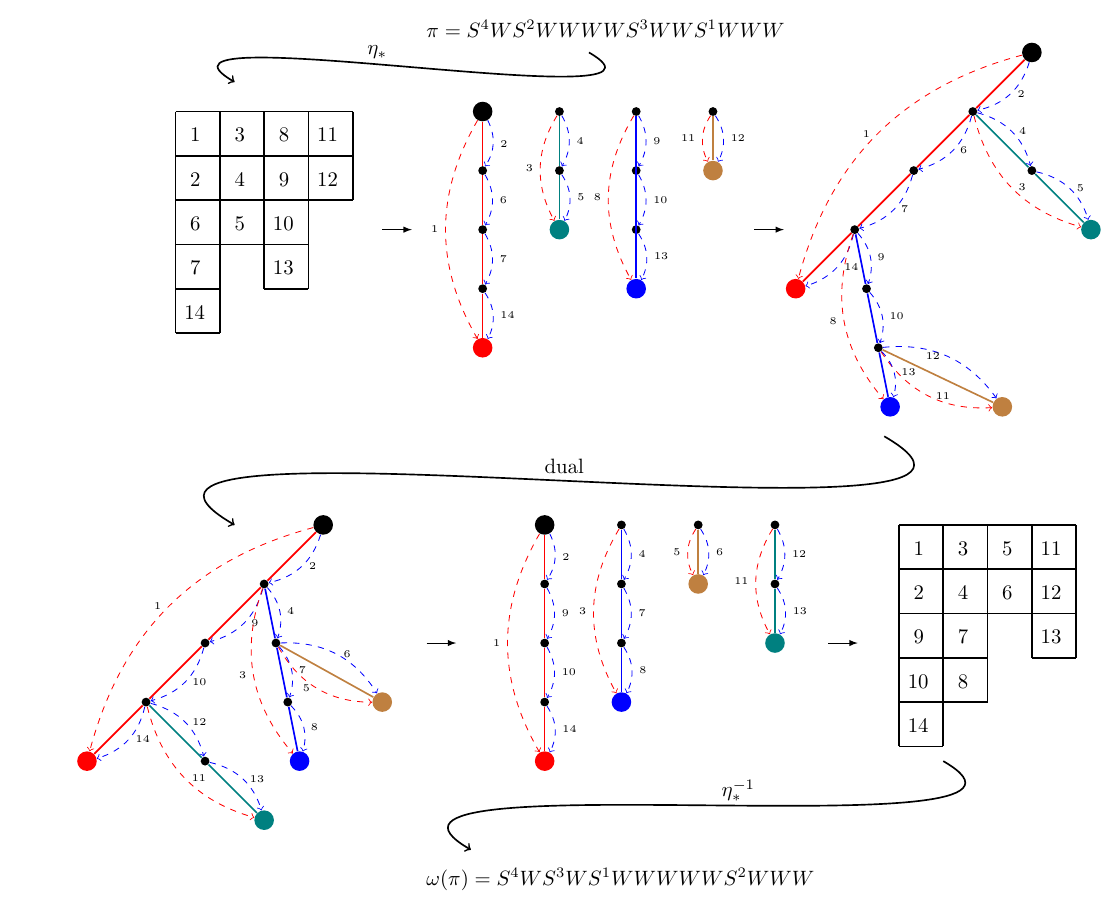}
    \caption{The construction of our involution $\omega$.}
    \label{fig:placeholder}
\end{figure}

\section{$q,t$-symmetry of $\widetilde{C}_{\vec{k}}(q,t)$}
Now, let us begin discussing the $q,t$-symmetry for singular $\mathcal{D}_{\vec{k}}$. Similar to the case of $C_{\vec{k}}(q,t)$, computational results indicate that $\widetilde{C}_{\vec{k}}(q,t)$ is generally not $q,t$-symmetric if $\ell(\vec{k})\geq 4$.

\begin{exa}
\begin{align*}
\widetilde{C}_{(1,1,2,1)}(q,t)-\widetilde{C}_{(1,1,2,1)}(t,q)&=-q^{4}t+qt^{4}+q^{3}t-qt^{3},\\
\widetilde{C}_{(1,2,1,1)}(q,t)-\widetilde{C}_{(1,2,1,1)}(t,q)&=q^{4}t-qt^{4}-q^{3}t+qt^{3}.
\end{align*}
\end{exa}

\subsection{$q,t$-symmetry of $\widetilde{C}_{\vec{k}}(q,t)$ for some special cases.}

\begin{cor}\label{cor-abb}
For any $\vec{k}=(a,b,\cdots,b,b)$ with positive integers $a$ and $b$, $\widetilde{C}_{\vec{k}}(q,t)$ is $q,t$-symmetric. In particular, $\widetilde{C}_{\vec{k}}(q,t)$ is $q,t$-symmetric if $\ell(\vec{k})=2$.
\end{cor}
\begin{proof}
It suffices to take $\mathcal{K}=\{(b,\cdots,b,b)\}$, and then the results follows from Theorem \ref{thm-aK}.
\end{proof}

We obtain a similar conclusion to Theorem \ref{thm-bk1k2}, which was proven in \cite{beck2024polyhedral}.

\begin{prop}\label{prop-k1k2}
For any two vectors $\vec{k}_1=(a_1,a_2,\cdots,a_{\ell},b)$ and $\vec{k}_2=(a_1,a_2,\cdots,a_{\ell},c)$ with positive integers $a_1, \cdots,a_{\ell},b$ and $c$, we have $\widetilde{C}_{\vec{k}_1}(q,t)=\widetilde{C}_{\vec{k}_2}(q,t)$.
\end{prop}
\begin{proof}
The proof closely follows that in \cite[Proposition 6.1, Corollary 6.2]{beck2024polyhedral}. To summarize, there is a bijection between $\mathcal{D}_{\vec{k}_1}$ and $\mathcal{D}_{\vec{k}_2}$, and depth statistic is also not affected by the last part of $\vec{k}$, which can be directly derived from the definitions of the Filling and Ranking algorithms.
\end{proof}

Thus, we can rephrase Corollary \ref{cor-abb} as follows:

\begin{cor}\label{cor-abbc}
For any $\vec{k}=(a,b,\cdots,b,c)$ with positive integers $a$, $b$, and $c$, $\widetilde{C}_{\vec{k}}(q,t)$ is $q,t$-symmetric. In particular, $\widetilde{C}_{\vec{k}}(q,t)$ is $q,t$-symmetric if $\ell(\vec{k})=3$.
\end{cor}

It is easy to see that $\omega$ action on $\mathcal{D}_{\vec{k}}$ with $\vec{k} = (a,b,\cdots,b,c)$ is not an involution.
However, we can obtain an involution on $\mathcal{D}_{\vec{k}}$ by $\omega$ through the following process:
Let $\vec{k'} = (a,b,\cdots,b,b)$,
$$\pi \in \mathcal{D}_{\vec{k}} \rightarrow \pi' \in \mathcal{D}_{\vec{k'}} \xrightarrow{\omega} \omega(\pi') \in \mathcal{D}_{\vec{k'}} \rightarrow \omega(\pi) \in \mathcal{D}_{\vec{k}}$$
where the Dyck path $\pi'$ (or $\omega(\pi')$ ) is the one-to-one correspondence with the Dyck path $\pi$ (or $\omega(\pi)$).
We will present such involution $\omega$, denoted by $\theta$, on $\mathcal{D}_{(a,b,c)}$ later. Additionally, we will also provide an algebraic proof using the method of MacMahon's Partition Analysis.

\subsection{Relationship between $C_{\vec{k}}(q,t)$ and $\widetilde{C}_{\vec{k}}(q,t)$.}

\begin{prop}\label{prop-k12}
If $l(\vec{k})=1$ or $2$, then for any $\pi\in \mathcal{D}_{\vec{k}}$, we have $\mathrm{depth}(\pi)=\mathrm{bounce}(\pi)$. In particular, we have
\begin{equation*}
\widetilde{C}_{\vec{k}}(q,t)={C}_{\vec{k}}(q,t).
\end{equation*}
\end{prop}
\begin{proof}
When $\ell(\vec{k})=1$, the claim is immediate. In this case there is a unique $\vec{k}$-Dyck path, given by $\pi=S^{|\vec{k}|}WW\cdots W$, for which $\mathrm{depth}(\pi)=\mathrm{bounce}(\pi)=0$.

If $\ell(\vec{k})=2$, we consider the construction process of the filling tableaux $\eta(\pi)$ and $\eta_{*}(\pi)$. The path $\pi$ can be encoded as
\begin{equation*}
\pi=S^a \underbrace{WW\cdots W}_{\ell_1 \ times}S^b \underbrace{WW\cdots W}_{a+b-\ell_1 \ times},
\end{equation*}
for some $0\leq \ell_1\leq a$.
The first rows of $\eta(\pi)$ and $\eta_{*}(\pi)$ are both $(1,\ell_1+2)$. The sequence $2,\cdots,\ell_1+1$ appears directly below $1$ in the first column of both $\eta(\pi)$ and $\eta_{*}(\pi)$. Consequently, the first rows of $\gamma(\eta(\pi))$ and $\gamma_*(\eta_{*}(\pi))$ are both $(0,\ell_1)$. Therefore
\begin{equation*}
\mathrm{depth}(\pi)=\mathrm{bounce}(\pi)=\ell_1.
\end{equation*}
This concludes the proof.
\end{proof}

\begin{prop}\label{prop-a1c}
If $\vec{k}=(a,1,c)$ with positive integers  $a$ and $c$, then for any $\pi\in \mathcal{D}_{\vec{k}}$ we have $\mathrm{depth}(\pi)=\mathrm{bounce}(\pi)$ and hence
\begin{equation*}
\widetilde{C}_{\vec{k}}(q,t)={C}_{\vec{k}}(q,t).
\end{equation*}
\end{prop}
\begin{proof}
The idea is similar to the previous cases. We may encode $\pi$ as
\begin{equation*}
\pi=S^a \underbrace{WW\cdots W}_{\ell_1 \ times}S^1\underbrace{WW\cdots W}_{\ell_2 \ times}S^c \underbrace{WW\cdots W}_{a+1+c-\ell_1-\ell_2 \ times},
\end{equation*}
for some $0\leq \ell_1\leq a$ and $0\leq \ell_2$ satisfying the relation $\ell_1+\ell_2\leq a+1$.

For $\ell_1\geq 1$ and $\ell_2>1$, the first rows of $\eta(\pi)$ and $\eta_*(\pi)$ are both $(1,\ell_1+2,\ell_1+\ell_2+3)$. The first column of $\eta(\pi)$ is $(1,\cdots,\ell_1+1,\ell_1+3,\ell_1+5,\cdots,\ell_1+\ell_2+2,\cdots)$ while the first column of $\eta_*(\pi)$ is $(1,2,\cdots,\ell_1+1,\ell_1+4,\ell_1+5,\cdots,\ell_1+\ell_2+2,\cdots)$. Thus the first rows of $\gamma(\eta(\pi))$ and $\gamma_*(\eta_*(\pi))$ are both $(0,\ell_1,\ell_1+\ell_2-1)$. Consequently,
\begin{equation*}
\mathrm{depth}(\pi)=\mathrm{bounce}(\pi)=\ell_1+\ell_1+\ell_2-1=2\ell_1+\ell_2-1.
\end{equation*}

For $\ell_1\geq 1$ and $\ell_2=1$, we have $$\mathrm{depth}(\pi)=\mathrm{bounce}(\pi)=2\ell_1+1.$$

Similarly, we have $\mathrm{depth}(\pi)=\mathrm{bounce}(\pi)=0$ for $\ell_1=\ell_2=0$; $\mathrm{depth}(\pi)=\mathrm{bounce}(\pi)=1$ for $\ell_1=0$ and $\ell_2= 1$; $\mathrm{depth}(\pi)=\mathrm{bounce}(\pi)=\ell_2-1$ for $\ell_1=0$ and $\ell_2>1$; $\mathrm{depth}(\pi)=\mathrm{bounce}(\pi)=2\ell_1$ for $\ell_1\geq 1$ and $\ell_2=0$.

This completes the proof.
\end{proof}

\subsection{$q,t$-symmetry of $\widetilde{C}_{\vec{k}}(q,t)$ for $\vec{k}=(a,b,c)$}.

We describe two additional methods that yield the $q,t$-symmetry of $\widetilde{C}_{(a,b,c)}(q,t)$.

\subsubsection{Proof via an explicit involution}
We aim to construct a direct involution on $\mathcal{D}_{(a,b,c)}$ to reveal its $q,t$-symmetry, as our $\omega$ in Section 4 might map $\pi\in \mathcal{D}_{(a,b,c)}$ to $\omega(\pi)\in \mathcal{D}_{(a,c,b)}$.

\begin{prop}
Let $\pi \in \mathcal{D}_{(a,b,c)}$.Then each $\pi$ can be encoded as
\begin{equation*}
\pi=S^{a}\underbrace{WW\cdots W}_{\ell_1\ times}S^{b}\underbrace{WW\cdots W}_{\ell_2\ times}S^{c}\underbrace{WW\cdots W}_{a+b+c-\ell_1-\ell_2\ times},
\end{equation*}
for some integers $l_1,l_2$ satisfying $0\leq \ell_1 \leq a$, $0\leq \ell_2$, and $\ell_1 + \ell_2 \leq a+b$.

Define the map $\theta$ on $\mathcal{D}_{(a,b,c)}$ by
\begin{equation*}
\theta(\pi)=
\begin{cases}
S^{a}\underbrace{WW\cdots W}_{a-\ell_1\ times}S^{b}\underbrace{WW\cdots W}_{b-\ell_2\ times}S^{c}\underbrace{WW\cdots W}_{c+\ell_1+\ell_2\ times}, & \text{if } \ell_2\leq b,\\
S^{a}\underbrace{WW\cdots W}_{a+b-\ell_1-\ell_2\ times}S^{b}\underbrace{WW\cdots W}_{\ell_2\ times}S^{c}\underbrace{WW\cdots W}_{c+\ell_1\ times}, & \text{otherwise.}
\end{cases}
\end{equation*}
Then $\theta$ is an involution on $\mathcal{D}_{(a,b,c)}$ such that
\begin{equation*}
\mathrm{area}(\pi)=\mathrm{depth}(\theta(\pi)), \ \mathrm{depth}(\pi)=\mathrm{area}(\theta(\pi)).
\end{equation*}
\end{prop}
\begin{proof}
It is straightforward to verify that $\theta$ is well-defined and that $\theta^2(\pi) = \pi$. Next, we analyze the contents of the filling tableaux and the ranking tableaux, and proceed by considering cases.

First, note that the area of $\pi$ is fixed:
\begin{equation*}
\mathrm{area}(\pi)=0+a-\ell_1+a+b-\ell_1-\ell_2=2a+b-2\ell_1-\ell_2.
\end{equation*}

If $\ell_2 \leq b$, then the first row of $\eta(\pi)$ is
$$(1,\, \ell_1+2,\, \ell_1+\ell_2+3).$$
When $1\leq \ell_1 \leq a$ and $1\leq \ell_2 \leq b$, the entries are arranged as follows:
$2,\dots,\ell_1+1$ are placed directly below $1$ in the first column;
$\ell_1+3,\dots,\ell_1+\ell_2+2$ are placed directly below $l_1+2$ in the second column;
$\ell_1+\ell_2+4,\dots,\ell_1+\ell_2+3+c$ are placed directly below $\ell_1+\ell_2+3$ in the third column.
The remaining $b-\ell_2$ entries are filled in the rest of the second column from top to bottom, and the last $a-l_1$ entries are filled in the rest of the first column from top to bottom.
Thus the first row of the ranking tableau $\gamma_*(\eta_*(\pi))$ is $(0,\ell_1,\ell_1+\ell_2)$. Hence
\begin{equation}\label{equ-depthabc}
\mathrm{depth}(\pi)=0+\ell_1+\ell_1+\ell_2=2\ell_1+\ell_2.
\end{equation}
Moreover, we have the following cases for $\mathrm{depth}(\pi)$: $\mathrm{depth}(\pi)=0$ when $\ell_1=\ell_2=0$, $\mathrm{depth}(\pi)=\ell_2$ when $\ell_1=0$ and $\ell_2\geq 1$, $\mathrm{depth}(\pi)=2\ell_1$ when $\ell_1\geq 1$ and $\ell_2=0$. These cases are consistent with Equation \eqref{equ-depthabc}.

Similarly, we compute the area and depth of $\theta(\pi)$. We have
\begin{equation*}
\mathrm{area}(\theta(\pi))=0+a-(a-\ell_1)+a+b-(a-\ell_1+b-\ell_2)=2\ell_1+\ell_2.
\end{equation*}
Moreover, the first row of $\eta_*(\theta(\pi))$ is $(1,a-\ell_1+2,a-\ell_1+b-\ell_2+3)$ and so the first row of $\gamma_*(\eta_*(\theta(\pi)))$ is $(0,a-\ell_1,a-\ell_1+b-\ell_2)$. Therefore, performing the same analysis yields
\begin{equation*}
\mathrm{depth}(\theta(\pi))=0+a-\ell_1+a-\ell_1+b-\ell_2=2a+b-2\ell_1-\ell_2.
\end{equation*}
This concludes the analysis of the case $\ell_2\leq b$.

Otherwise, we have
\begin{equation*}
\mathrm{area}(\theta(\pi))=0+a-(a+b-\ell_1-\ell_2)+a+b-(a+b-\ell_1-\ell_2+\ell_2)=2\ell_1+\ell_2-b.
\end{equation*}
To compute $\mathrm{depth}(\pi)$, note that the first row of $\eta_*(\pi)$ is $(1,\ell_1+2,\ell_1+\ell_2+3)$. Since $\ell_1 \leq a$ and $\ell_2 > b$, the entries are arranged as follows: $2,\dots,\ell_1+1$ are placed directly below $1$ in the first column;
$\ell_1+3,\dots,\ell_1+2+b$ are placed directly below $\ell_1+2$ in the second column;
$\ell_1+3+b,\dots,\ell_1+\ell_2+2$ are then placed below $\ell_1+1$ in the first column;
$\ell_1+\ell_2+4,\dots,\ell_1+\ell_2+3+c$ are placed directly below $\ell_1+\ell_2+3$ in the third column;
finally, the remaining entries are filled in the rest of the first column from top to bottom.
Thus, the first row of the ranking tableau $\gamma_*(\eta_*(\pi))$ is $(0,\ell_1,\ell_1+\ell_2-b)$ and hence
\begin{equation}
\mathrm{depth}(\pi)=0+\ell_1+\ell_1+\ell_2-b=2\ell_1+\ell_2-b.
\end{equation}
Similarly, the first row of the ranking tableau $\gamma_*(\eta(\eta_*(\pi)))$ is $(0,a-(\ell_1+\ell_2-b),a-(\ell_1+\ell_2-b)+\ell_2-b)$ and therefore
\begin{equation*}
\mathrm{depth}(\theta(\pi))=0+a-(\ell_1+\ell_2-b)+a-(\ell_1+\ell_2-b)+\ell_2-b=2a+b-2\ell_1-\ell_2.
\end{equation*}
This concludes the analysis of the case $\ell_2> b$.

Hence, the proof is complete.
\end{proof}

\subsubsection{Proof via MacMahon's Partition Analysis}

In our analysis of the previous proposition, we obtained explicit formulas for $\mathrm{area}(\pi)$ and $\mathrm{depth}(\pi)$, where
\begin{equation*}
\pi=S^{a}\underbrace{WW\cdots W}_{\ell_1\ times}S^{b}\underbrace{WW\cdots W}_{\ell_2\ times}S^{c}\underbrace{WW\cdots W}_{a+b+c-\ell_1-\ell_2\ times},
\end{equation*}
for some $0\leq \ell_1 \leq a$ and $0\leq l_2$ satisfying $\ell_1+\ell_2 \leq a+b$.
These are given by
\begin{align*}
\mathrm{area}(\pi)&=2a+b-2\ell_1-\ell_2,\\
\mathrm{depth}(\pi)&=
\begin{cases}
2\ell_1+\ell_2, & \text{if } \ell_2\leq b,\\
2\ell_1+\ell_2-b, & \text{otherwise.}
\end{cases}
\end{align*}

Consequently, it is natural to analyze the generating function within the framework of MacMahon’s partition analysis \cite{andrews2001macmahon}. Define

\begin{equation*}
\underset{\geq}{\Omega}\sum_{i_1=-\infty}^{\infty}\cdots\sum_{i_r=-\infty}^{\infty}A_{i_1,\cdots,i_r}\lambda_1^{i_1}\lambda_{2}^{i_2}\cdots \lambda_{r}^{i_r}:=\sum_{i_1=0}^{\infty}\cdots\sum_{i_r=0}^{\infty}A_{i_1,\cdots,i_r}.
\end{equation*}

It means that $\Omega_{\geq}$ operator extracts all terms with nonnegative power in $\lambda_1$, $\lambda_2$, $\cdots$, $\lambda_r$ and sets them to be equal to 1. We have the following crude generating function.

\begin{align}\label{equ-gf}
F(x_1,x_2,x_3,y_1,y_2,q,t)=\sum_{a,b,c\geq 0}x_1^{a}x_{2}^{b}x_{3}^{c}\sum_{0\leq \ell_1\leq a, 0\leq \ell_2, 0\leq \ell_1+\ell_2\leq a+b}q^{2a+b-2l_1-l_2}t^{depth(\pi)}y_1^{\ell_1}y_2^{\ell_2}.
\end{align}

It is sufficient to prove the $q,t$-symmetry of
\begin{align*}
F(x_1,x_2,x_3,1,1,q,t)=\sum_{a,b,c\geq 0}x_1^{a}x_2^bx_3^c\widetilde{C}_{(a,b,c)}(q,t).
\end{align*}
We divide Equation (\ref{equ-gf}) into two parts for further analysis.
\begin{align}
F_1(x_1,x_2,x_3,y_1,y_2,q,t):&=\sum_{a,b,c\geq 0}x_1^{a}x_{2}^{b}x_{3}^{c}\sum_{0\leq \ell_1\leq a, 0\leq \ell_1+\ell_2\leq a+b,0\leq \ell_2\leq b}q^{2a+b-2\ell_1-\ell_2}t^{2\ell_1+\ell_2}y_1^{\ell_1}y_2^{\ell_2},\\
F_2(x_1,x_2,x_3,y_1,y_2,q,t):&=\sum_{a,b,c\geq 0}x_1^{a}x_{2}^{b}x_{3}^{c}\sum_{0\leq \ell_1\leq a, 0\leq \ell_1+\ell_2\leq a+b,\ell_2> b}q^{2a+b-2\ell_1-\ell_2}t^{2\ell_1+\ell_2-b}y_1^{\ell_1}y_2^{\ell_2}.
\end{align}

We use the Maple package Ell by Xin \cite{xin2004fast} for the calculations, and the results yield

\begin{align*}
F_1&=\sum_{a,b,c,\ell_1,\ell_2\geq 0}x_1^{a}x_{2}^{b}x_{3}^{c}\sum_{a-\ell_1\geq 0, a+b-\ell_1-\ell_2\geq 0, b-\ell_2\geq 0}q^{2a+b-2\ell_1-\ell_2}t^{2\ell_1+\ell_2}y_1^{\ell_1}y_2^{\ell_2}\\
&=\underset{\geq}{\Omega} \sum_{a,b,c,\ell_1,\ell_2\geq 0}x_1^{a}x_{2}^{b}x_{3}^{c}q^{2a+b-2\ell_1-\ell_2}t^{2\ell_1+\ell_2}y_1^{\ell_1}y_2^{\ell_2}\lambda_{1}^{a-\ell_1}\lambda_{2}^{a+b-\ell_1-\ell_2}\lambda_{3}^{b-\ell_2}\\
&=\underset{\geq}{\Omega} \frac{1}{(1-q^{2}x_1\lambda_1\lambda_2)(1-qx_2\lambda_2\lambda_3)(1-x_3)(1-\frac{t^2y_1}{q^2\lambda_1\lambda_2})(1-\frac{ty_2}{q\lambda_2\lambda_3})}\\
&=\frac{1}{(1-q^{2}x_1)(1-qx_2)(1-x_3)(1-t^{2}x_1y_1)(1-tx_2y_2)}.
\end{align*}

\begin{align*}
F_2&=\sum_{a,b,c,\ell_1,\ell_2\geq 0}x_1^{a}x_{2}^{b}x_{3}^{c}\sum_{a-\ell_1\geq 0, a+b-\ell_1-\ell_2\geq 0, \ell_2-b-1\geq 0}q^{2a+b-2\ell_1-\ell_2}t^{2\ell_1+\ell_2-b}y_1^{\ell_1}y_2^{\ell_2}\\
&=\underset{\geq}{\Omega} \sum_{a,b,c,\ell_1,\ell_2\geq 0}x_1^{a}x_{2}^{b}x_{3}^{c}q^{2a+b-2\ell_1-\ell_2}t^{2\ell_1+\ell_2-b}y_1^{\ell_1}y_2^{\ell_2}\lambda_{1}^{a-\ell_1}\lambda_{2}^{a+b-\ell_1-\ell_2}\lambda_{3}^{\ell_2-b-1}\\
&=\underset{\geq}{\Omega} \frac{\lambda_3^{-1}}{(1-q^{2}x_1\lambda_1\lambda_2)(1-\frac{qx_2\lambda_2}{t\lambda_3})(1-x_3)(1-\frac{t^2y_1}{q^2\lambda_1\lambda_2})(1-\frac{ty_2\lambda_3}{q\lambda_2})}\\
&=\frac{qtx_1y_2}{(1-q^{2}x_1)(1-x_2y_2)(1-x_3)(1-t^{2}x_1y_1)(1-qtx_1y_2)}.
\end{align*}

Therefore, by substituting
$y_1$ and $y_2$ with $1$ in the expressions above and combining the results, we obtain
\begin{align*}
F(x_1,x_2,x_3,1,1,q,t)
&=F_1(x_1,x_2,x_3,1,1,q,t)+F_2(x_1,x_2,x_3,1,1,q,t)\\
&=\frac{1-x_2+qtx_1x_2-q^2tx_1x_2-qt^2x_1x_2+q^2t^2x_1x_2^2}{(1-q^2x_1)(1-t^2x_1)(1-qtx_1)(1-x_2)(1-qx_2)(1-tx_2)(1-x_3)}.
\end{align*}
The $q,t$-symmetry clearly follows from this.

\section{Conclusion and future directions}

In this article, we establish the $q,t$-symmetry of $\widetilde{C}_{\mathcal{K}}(q,t)$, which is the $q,t$-polynomial graded by the pair of statistics (area,depth) on $\mathcal{K}$-Dyck paths. Our proof relies on constructing an involution on $\mathcal{K}$-Dyck paths, which swaps the area and depth of a path. However, this involution cannot be used to prove a similar result for $C_{\mathcal{K}}(q,t)$, as it is not generally $q,t$-symmetric. Additionally, we analyze the $q,t$-symmetry in the refined case for certain singular $\vec{k}$. The dinv of a $\vec{k}$-Dyck path $\pi$ is not defined using the area sequence of $\pi$, and thus, we do not identify a direct generalization of ddinv for $\vec{k}$-Dyck paths via the depth labeling sequence. However, using the $\mathrm{depth}(\pi)$ and $\mathrm{dinv}(\omega(\pi))$ statistics, can also provide an alternative description of the higher $q,t$-Catalan polynomials.

Similar to the development of the $q,t$-symmetry of $C_{\vec{k}}(q,t)$, in \cite{beck2024polyhedral}, some results from \cite{xin2025q} are reproven using a different perspective. Therefore, one possible direction for further research is:

\begin{prob}
Explore the $q,t$-symmetry of $\widetilde{C}_{\vec{k}}(q,t)$ and $\widetilde{C}_{\mathcal{K}}(q,t)$ using techniques from polyhedral geometry.
\end{prob}

Despite the pair of statistics discussed above, one might wonder if there exist other pairs of statistics exhibit $q,t$-symmetry on $\vec{k}$-Dyck paths or $\mathcal{K}$-Dyck paths. The answer appears to be obvious. We present two types of new $q,t$-polynomials here.

In \cite{li2023symmetry}, the authors demonstrated that the pair (run, ret) constitutes a $q,t$-symmetric pair of statistics on \textbf{classical Dyck paths of composition type $\alpha$}, meaning that for each $\pi\in \mathcal{D}_{n}$, the lengths of successive North-step runs are determined by the composition $\alpha\vDash n$ in left-to-right order. In the context of $\vec{k}$-Dyck paths, \textbf{run} represents the sum of the lengths of all $S^*$ segments occurring before the first $WW$ in $\pi$, while \textbf{ret} counts the number of times the path, excluding $(0,0)$, intersects the horizontal axis. The distinction between $\mathcal{D}_{n}$ of composition type $\vec{k}$ and $\vec{k}$-Dyck paths lies in the fact that, for $\vec{k}$-Dyck paths, successive North-step runs may occur immediately above one another. Consequently, the $q,t$-polynomial associated with $\vec{k}$-Dyck paths, when graded by the pair (run, ret), can be interpreted as a summation over appropriate families of classical Dyck paths of composition type. The $q,t$-symmetry of these polynomials follows directly from the result in \cite{li2023symmetry}.

To summarize, for any pair of statistics $(\mathrm{stat1},\mathrm{stat2})$, if it is $q,t$-symmetric on classical Dyck paths of composition type $\alpha$, then it is also $q,t$-symmetric on $\vec{k}$-Dyck paths and $\mathcal{K}$-Dyck paths. Therefore, our next direction is as follows.

\begin{prob}
Find additional pairs of $q,t$-symmetric statistics $(\mathrm{stat1}, \mathrm{stat2})$ on classical Dyck paths of composition type $\alpha$.
\end{prob}

Finally, the results in Proposition \ref{prop-k12} and \ref{prop-a1c} suggest that the $q,t$-polynomial $\widehat{C}_{\vec{k}}(q,t)$, graded by the pair of statistics (bounce, depth) on $\vec{k}$-Dyck paths, may also exhibit $q,t$-symmetry for some special cases. For instance, $\widehat{C}_{\vec{k}}(q,t)$ is $q,t$-symmetric for $\ell(\vec{k})\leq 2$ or $\vec{k}=(a,1,c)$ with positive $a$ and $c$, as in these cases, $\mathrm{bounce}(\pi)=\mathrm{depth}(\pi)$ for every $\pi$.

Based on computational data, we make the following observation.

\begin{prob}
Prove the following observation: If $\vec{k}=(a,2,c)$, $(a,1,1,d)$, $(a,2,1,d)$, or $(a,1,1,1,e)$, where $c$, $d$, and $e$ are positive, then $\widehat{C}_{\vec{k}}(q,t)$ is $q,t$-symmetric. Moreover, $\widehat{C}_{\vec{k}}(q,t)$ does not exhibit $q,t$-symmetric if $\ell(\vec{k})\geq 6$.
\end{prob}

This problem may be approached using techniques from MacMahon's partition analysis, specifically by applying explicit formulas for bounce and depth. We leave it as an exercise for interested readers.

\noindent
\textbf{Acknowledgements:} The authors would like to thank the anonymous referee for valuable suggestions for improving the presentation. Y. Zhang was supported by the Yunnan Provincial Department of Education Scientific Research Fund Project Grant: 2026J0603.

\bibliographystyle{amsalpha}

\end{document}